\documentclass[reqno,a4paper,11pt]{article}

\usepackage{amsmath,amssymb,amsthm}

\usepackage{mathrsfs}

\usepackage{mathtools}
\usepackage{commath}

\usepackage{thmtools}
\usepackage{thm-restate}

\usepackage{cases}
\usepackage{enumitem}
\setlist[enumerate,1]{label=(\roman*)}

\usepackage[pdftex, pdfborderstyle={/S/U/W 0}]{hyperref}
\hypersetup{
    colorlinks=true,
    linkcolor=red,
    citecolor=blue,
}

\usepackage{cleveref}

\usepackage{etoolbox}

\usepackage{comment}

\usepackage[numbers]{natbib}

\numberwithin{equation}{section}

\ifdefined\thmcolor
\declaretheoremstyle[
  shaded={bgcolor=\thmcolor}
]{plain}
\else
\fi

\ifdefined\defcolor
\declaretheoremstyle[
  headfont=\normalfont\bfseries,
  bodyfont=\normalfont,
  shaded={bgcolor=\defcolor}
]{noital}
\else
\declaretheoremstyle[
  headfont=\normalfont\bfseries,
  bodyfont=\normalfont,
]{noital}
\fi

\declaretheorem[style=plain,numberwithin=section,name=Theorem]{theorem}

\declaretheorem[style=plain,sibling=theorem,name=Lemma]{lemma}

\declaretheorem[style=plain,sibling=theorem,name=Conjecture]{conjecture}

\declaretheorem[style=plain,sibling=theorem,name=Question]{question}

\declaretheorem[style=plain,numbered=no,name=Theorem]{theorem-n}
\declaretheorem[style=plain,numbered=no,name=Proposition]{proposition-n}
\declaretheorem[style=plain,numbered=no,name=Lemma]{lemma-n}
\declaretheorem[style=plain,numbered=no,name=Corollary]{corollary-n}
\declaretheorem[style=plain,numbered=no,name=Conjecture]{conjecture-n}
\declaretheorem[style=plain,numbered=no,name=Claim]{claim-n}
\declaretheorem[style=plain,numbered=no,name=Fact]{fact-n}
\declaretheorem[style=plain,numbered=no,name=Open Problem]{openproblem-n}
\declaretheorem[style=plain,numbered=no,name=Question]{question-n}
\declaretheorem[style=plain,numbered=no,name=Observation]{observation-n}

\declaretheorem[style=noital,numbered=no,name=Remark]{remark-n}
\declaretheorem[style=noital,numbered=no,name=Definition]{definition-n}
\declaretheorem[style=noital,numbered=no,name=Construction]{construction-n}
\declaretheorem[style=noital,numbered=no,name=Example]{example-n}

\newcommand{\defined}{\mathrel{\coloneqq}}

\newcommand{\st}{\mathbin{\colon}}

\undef{\set}
\DeclarePairedDelimiter{\set}{\lbrace}{\rbrace}

\undef{\emptyset}
\newcommand{\emptyset}{\varnothing}

\DeclarePairedDelimiter{\card}{\lvert}{\rvert}

\newcommand{\union}{\mathbin{\cup}}

\newcommand{\bigunion}{\mathbin{\bigcup}}

\newcommand{\from}{\colon}

\newcommand{\setm}[1]{\setminus\set{#1}}

\undef{\mod}
\newcommand{\mod}[1]{\ (\mathrm{mod}\ #1)}

\undef{\abs}
\DeclarePairedDelimiterX{\abs}[1]
  {\lvert}{\rvert}{\ifblank{#1}{\,\cdot\,}{#1}}

\undef{\norm}
\DeclarePairedDelimiterX{\norm}[1]
  {\lVert}{\rVert}{\ifblank{#1}{\,\cdot\,}{#1}}

\DeclarePairedDelimiterX{\inner}[2]
  {\langle}{\rangle}{\ifblank{#1}{\,\cdot\,}{#1},\ifblank{#2}{\,\cdot\,}{#2}}

\DeclarePairedDelimiterX{\absinner}[2]
  {|\langle}{\rangle|}{\ifblank{#1}{\,\cdot\,}{#1},\ifblank{#2}{\,\cdot\,}{#2}}

\DeclareMathDelimiter{\given}
  {\mathbin}{symbols}{"6A}{largesymbols}{"0C}

\DeclareMathOperator{\Prob}{\mathbb{P}}
\DeclarePairedDelimiterXPP{\prob}[1]
  {\Prob}{\lparen}{\rparen}{}
  {\renewcommand{\given}{\nonscript\;\delimsize\vert\nonscript\;\mathopen{}}#1}

\DeclareMathOperator{\Expec}{\mathbb{E}}
\DeclarePairedDelimiterXPP{\expec}[1]
  {\Expec}{\lparen}{\rparen}{}
  {\renewcommand{\given}{\nonscript\;\delimsize\vert\nonscript\;\mathopen{}}#1}

\DeclareMathOperator{\Var}{Var}
\DeclarePairedDelimiterXPP{\var}[1]
  {\Var}{\lparen}{\rparen}{}
  {\renewcommand{\given}{\nonscript\;\delimsize\vert\nonscript\;\mathopen{}}#1}

\DeclareMathOperator{\Cov}{Cov}
\DeclarePairedDelimiterXPP{\cov}[2]
  {\Cov}{\lparen}{\rparen}{}{#1,#2}

\newcommand{\sseq}{\subseteq}

\newcommand{\NN}{\mathbb{N}}

\newcommand{\RR}{\mathbb{R}}

\usepackage{geometry}
\usepackage{tikz}
\usepackage{float}

\usepackage{scrextend}

\usepackage[T1]{fontenc}
\usepackage{titlesec}

\titleformat{\section}{\centering\bfseries\scshape\Large}{\thesection}{1em}{}
\titleformat{\subsection}{\bfseries\scshape\large}{\thesubsection}{1em}{}

\begin{document}

\title{\textsc{\bfseries Counterexamples to conjectures on strong maximality and minimality}}

\renewcommand{\thefootnote}{\fnsymbol{footnote}}

\author{\textsc{Lawrence Hollom}\footnotemark[1] \and \textsc{Benedict Randall Shaw}\footnotemark[1]}

\footnotetext[1]{\href{mailto:lh569@cam.ac.uk}{lh569@cam.ac.uk} and \href{mailto:bwr26@cam.ac.uk}{bwr26@cam.ac.uk} respectively. Department of Pure Mathematics and Mathematical Statistics (DPMMS), University of Cambridge, Wilberforce Road, Cambridge, CB3 0WA, United Kingdom}

\renewcommand{\thefootnote}{\arabic{footnote}}

\maketitle

\date{}

\begin{abstract}
  We provide counterexamples to several conjectures concerning strongly maximal and strongly minimal structures in infinite graphs and hypergraphs.
  In particular, we construct 3-uniform hypergraphs without strongly maximal matchings and without strongly minimal covers, and from our construction for covers we build a graph with no strongly minimal colouring.
  We also consider several refinements of these problems.


  Our results resolve conjectures and questions of Aharoni; Aharoni and Berger; Aharoni, Berger, Georgakopoulos, and Sprüssel; Aharoni and Korman; and Tardos.
\end{abstract}


\section{Introduction}

Many problems in graph theory ask when one may find a certain object of the greatest possible size. 
For example, Hall's theorem gives a condition for the existence of a matching which covers every vertex, and K\"{o}nig's theorem tells us that the largest matching and the smallest vertex-cover (that is, a set of vertices of the graph which meet every edge) have the same size.

When moving to the infinite case, versions of these results concerning only cardinality tend to be weak, and easy to prove. 
For example, given Hall's condition in an infinite bipartite graph, one can simply apply a greedy algorithm to construct a matching with the same cardinality as the vertex set of the graph.
However, the fact that the object---a matching, in this case---attains the maximum possible size tells us very little about its structure: in contrast with the finite case, it may not even be maximal. 
However, although in fact we may show a maximal matching exists, even that may not capture adequately the notion of being a `largest' object.

One way around these issues is to strengthen our desired notion of `maximality' to `strong maximality', which informally means that we can make no local improvement to our object.
More precisely, a \emph{strongly maximal matching} of a graph \(G\) is a matching \(M\) such that, for every matching \(M'\) of \(G\), we have \(\card{M \setminus M'} \geq \card{M' \setminus M}\).
Similarly, a \emph{strongly minimal vertex-cover} of \(G\) is a vertex-cover \(C\) such that, for every vertex-cover \(C'\) of \(G\), we have \(\card{C \setminus C'} \leq \card{C' \setminus C}\).

The notion of strong maximality more accurately captures our intuition of a `largest' object.
Returning to K\"{o}nig's theorem, it was proved by Aharoni \cite{Aha84} in 1984, confirming a conjecture of Erd\H{o}s (see e.g.\ \cite{Nas67}) that in any bipartite graph \(G\) one may find a strongly maximal matching \(M\) and a strongly minimal vertex-cover \(C\), and furthermore \(M\) and \(C\) are orthogonal---that is, one may find a bijection between \(C\) and \(M\) such that the vertex of \(C\) is contained in the corresponding edge of \(M\).
In fact, strong maximality and minimality of these objects follows from orthogonality alone, much like in the finite case.

Following on from this result, it is natural to then ask whether hypergraphs must also have strongly maximal matchings and strongly minimal covers. Here a \emph{hypergraph} \(H\) comprises a vertex set \(V\) together with a set \(E\) of edges, each of which is a subset of \(V\). We say that \(H\) is \(k\)-\emph{uniform} if every edge has size \(k\). Then a \emph{matching} is a subset of \(E\) whose edges are pairwise disjoint; an \emph{edge-cover} is a subset of \(E\) whose union is \(V\); and a \emph{vertex-cover} is a set of vertices which has non-empty intersection with every edge of \(E\).

Indeed, Aharoni \cite[Problem 5.5]{Aha91} asked in 1991 whether every hypergraph with finite edges contains a strongly maximal matching, and in 1992, Aharoni and Korman \cite[Conjecture 5.4]{AK92} conjectured that, in the absence of isolated vertices, there should also be a strongly minimal edge-cover.
Erd\H{o}s \cite{Erd94} expected the conjecture to be false and, indeed, counterexamples were found for both matchings and edge-covers by Ahlswede and Khachatrian \cite{AK96}, and van der Zypen \cite{vdZ22} respectively, in results which we shall return to later.
Nevertheless, these conjectures were later refined by Aharoni, Berger, Georgakopoulos, and Spr\"{u}ssel \cite[Conjecture 1.4]{ABGS08} to the following.

\begin{conjecture}
\label{conj:matching-and-cover}
    In any hypergraph with finitely bounded size of edges and no isolated vertices there exist a strongly maximal matching and a strongly minimal edge-cover.
\end{conjecture}

Here the requirement that there be no isolated vertices is simply to ensure that there is any edge-cover at all.
Our first two results show that both parts of \Cref{conj:matching-and-cover} are false.

\begin{theorem}\label{thm:SMmatching}
    There is a $3$-uniform hypergraph $H_1$ with no strongly maximal matching.
\end{theorem}

\begin{theorem}\label{thm:SMedgecover}
    There is a $3$-uniform hypergraph $H_2$ with no isolated vertices and no strongly minimal edge-cover.
\end{theorem}

We remark here that, by taking a disjoint union of \(H_1\) and \(H_2\), we obtain a 3-uniform hypergraph which has neither a strongly maximal matching not a strongly minimal edge-cover.

Covers and matchings are not the only objects for which one might consider strong minimality and maximality.
Indeed, Aharoni, Berger, Georgakopoulos, and Spr\"{u}ssel \cite{ABGS08} also highlighted the following problem.

\begin{conjecture}
\label{conj:colouring}
    In every graph there exists a strongly minimal cover of the vertex set by independent sets.
\end{conjecture}

The structure posited to exist by \Cref{conj:colouring} is referred to by Aharoni and Berger in \cite{AB11} as a \emph{strongly minimal colouring}, in that such a strongly minimal cover would be a proper vertex-colouring of the graph satisfying a strongly minimal condition on the colour classes.

Problems of this form were also considered in Aharoni and Berger's seminal paper \cite{AB09} proving Menger's theorem for infinite graphs.
Indeed, alongside \Cref{conj:matching-and-cover}, they also conjectured that every \emph{flag complex}---a hypergraph which is down-closed (every subset of an edge is an edge) and 2-determined (if all 2-element subsets of a set are edges, then so is that set)---has a strongly minimal edge-cover.
We show that the above problem of flag complexes is equivalent to \Cref{conj:colouring}, and thus deduce the following result.

\begin{theorem}\label{thm:SMcolouring}
    There is a graph $G$ with no strongly minimal colouring.
\end{theorem}

Aharoni and Berger also stated the following problem, attributed to Tardos (see \cite[Problem 10.4]{AB09}).

\begin{question}
\label{q:tardos}
    Is it true that, in every hypergraph with finite edges, there exists a matching \(M\) such that no matching \(M'\) exists for which \(\card{M\setminus M'} = 1\) and \(\card{M' \setminus M} = 2\)?
\end{question}

Our last result resolves \Cref{q:tardos} in the negative.

\begin{theorem}\label{thm:tardos}
    There is a hypergraph \(H_3\) such that, for every matching \(M\) of \(H_3\), there is a matching \(M'\) of \(H_3\) with \(\card{M\setminus M'} = 1\) and \(\card{M' \setminus M} = 2\).
\end{theorem}

Finally, \Cref{conj:matching-and-cover} was also repeated in \cite{Aha22}, where it was also conjectured that every hypergraph has a strongly minimal vertex-cover, but this is easily seen to be false, as shown in \Cref{subsec:SMvertexcover}.

The constructions for \Cref{thm:SMmatching,thm:SMedgecover} both begin with the counterexamples found for the case in which edges may have unbounded size, which are presented in \Cref{subsec:known}.
Both constructions work by replacing the edges in the known counterexamples with a `gadget', which simulates the large edges with only 3-edges and 2-edges, before the 2-edges are later replaced with 3-edges.
This gadget is described in \Cref{sec:gadget}.


\subsection{Known counterexamples for unbounded edge size}
\label{subsec:known}

As discussed in the introduction, \Cref{conj:matching-and-cover} was initially stated for strongly maximal matchings and strongly minimal covers in hypergraphs without the bound on the size of the edges, but these were shown to be false by Ahlswede and Khachatrian and by van der Zypen respectively. 
For completeness, we include these counterexamples.
We first consider the case of strongly maximal matchings, due to Ahlswede and Khachatrian \cite{AK96}.

\begin{theorem}
\label{thm:SMmatching_unbdd}
    Let $H_1^*$ be the hypergraph on vertex set $\NN$ with edge set consisting of those finite subsets $E\sseq \NN$ for which $\card{E} = \min E$.
    Then $H_1^*$ has no strongly maximal matching.
\end{theorem}

To see that $H_1^*$ has no strongly maximal matching, take some infinite matching $M$, and take $E_1,E_2\in M$ such that, letting $E_1=\set{v_1,v_2,\dotsc,v_t}$, we have $\min E_2 = \card{E_2} \geq v_2 + v_3$.
Then $\card{E_1\union E_2} = \card{E_1} + \card{E_2} \geq v_1 + v_2 + v_3$, and so we can find pairwise disjoint $F_1,F_2,F_3\sseq E_1\union E_2$ with $\card{F_i} = \min F_i = v_i$ for all $i=1,2,3$.

We may also note that $H_1^*$ also has no strongly minimal vertex-cover, as any vertex-cover $A\sseq \NN$ is cofinite, and if $\NN\setminus A = \set{v_1,\dotsc,v_n}$, then it is not hard to see that $(A\union\set{v_1})\setminus\set{v_n + 1, v_n + 2}$ is also a vertex-cover.

We next give van der Zypen's construction for strongly minimal edge-covers \cite{vdZ22}.

\begin{theorem}
\label{thm:SMcover_unbdd}
    Let $H_2^*$ be the hypergraph on vertex set $\NN$ with edge set consisting of all finite subsets of $\NN$.
    Then $H_2^*$ has no strongly minimal edge-cover.
\end{theorem}

To see that $H_2^*$ has no strongly minimal edge-cover, take some edge cover $C$ of $H_2^*$, and let $E_1,E_2\in C$ be two arbitrary edges.
Then $(C \setminus \set{E_1,E_2})\union\set{E_1\union E_2}$ is also an edge-cover of $H_2^*$.


\section{The gadget}
\label{sec:gadget}
We construct the hypergraphs of \Cref{thm:SMmatching,thm:SMedgecover} from the counterexamples in \Cref{subsec:known}, by `simulating' edges of size larger than \(3\) using a certain gadget with smaller edges. Given an edge \(e\), we now define the graph \(G_e\), the \emph{gadget on} \(e\).
First, fix an arbitrary labelling \(\{v_1,\dots,v_k\}\) of \(e\). We add \(2k-2\) new vertices in total, which will be used only in the gadget on \(e\): \(v_i^+=v_i^+(e)\) for \(1\leq i \leq k-1\), and \(v_i^-=v_i^-(e)\) for \(2\leq i \leq k\). The gadget \(G_e\) will have vertex set 
\[V\left(G_e\right)=e\cup\{v_i^+: 1\leq i \leq k-1\}\cup \{v_i^-:2\leq i \leq k\},\]
and edge set
\[E\left(G_e\right)=\{v_i^+v_{i+1}^-: 1\leq i \leq k-1\}\cup\{v_i^{\vphantom{+}}v_i^-v_i^+: 2\leq i \leq k-1\}\cup \{v_1^{\vphantom{+}} v_1^+ ,v_k^{\vphantom{+}} v_k^-\}.\]

We call the \(k\) edges incident to vertices of \(e\) the \emph{outer edges}, and call the remaining \(k-1\) edges the \emph{inner edges}, drawn in red in \Cref{fig:gadget}.

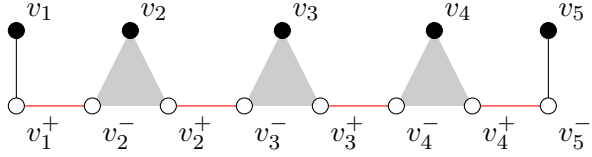
\begin{figure}[H]
\centering
    \begin{tikzpicture}
        \def\n{5}
        \def\nminusone{4}

        \foreach \x in {3,...,\nminusone} \draw[red] (2*\x-1.5,0) -- (2*\x-0.5,0);
        \foreach \x in {2,...,\nminusone}
        \fill[black!20!white] (2*\x,1) -- (2*\x+0.5,0) -- (2*\x-0.5,0) -- (2*\x,1);

        \filldraw[black] (2.5,1) circle (3pt) node[anchor=south west]{\(v_1\)};\filldraw[black] (2*\n-0.5,1) circle (3pt) node[anchor=south west]{\(v_\n\)};
        
        \foreach \x in {2,...,\nminusone}
        \filldraw[black] (2*\x,1) circle (3pt) node[anchor=south west]{\(v_\x\)};
        \draw (2*\n-0.5,1) -- (2*\n-0.5,0);
        \draw[red] (2*\n-0.5,0) -- (2*\n - 1.5, 0);
        \draw[red] (3.5,0) -- (2.5,0);
        \draw (2.5,0) -- (2.5,1);
        \draw[black,fill=white] (2.5,0) circle (3pt) node[anchor=north west]{\(v_1^+\)};
        \draw[black,fill=white] (2*\n-0.5,0) circle (3pt) node[anchor=north west]{\(v_\n^-\)};
        \foreach \x in {2,...,\nminusone} {
        \draw[black,fill=white] (2*\x-0.5,0) circle (3pt) node[anchor=north west]{\(v_\x^-\)};
        \draw[black,fill=white] (2*\x+0.5,0) circle (3pt) node[anchor=north west]{\(v_\x^+\)}; }
        
    \end{tikzpicture}
    \caption{The gadget \(G_e\) on the edge \(e=\{v_1,\dots,v_5\}\), with inner edges marked in red, and outer edges in grey and black.\label{fig:gadget}}
\end{figure}

A key observation is that the outer edges are a matching of \(G_e\) of size \(k\), and that this is the unique largest matching on \(G_e\). However, notice also that the inner edges form a matching of size \(k-1\) which does not meet any of the points of \(e\).

The gadget will be used by replacing every edge \(e\) of a hypergraph \(H^*\) with the gadget \(G_e\), to form a hypergraph in which every edge has size at most 3, no matter how large the edges were in \(H^*\).


\section{Proofs}
\label{sec:proofs}

We now produce our counterexamples in turn, proving \Cref{thm:SMmatching,thm:SMedgecover} in \Cref{subsec:matching,subsec:edgecover} respectively.
Then, in \Cref{subsec:colouring} \Cref{thm:SMcolouring} by means of constructing a flag complex with no strongly minimal edge-cover.
In \Cref{subsec:SMvertexcover}, we give a short construction of a graph with no strongly minimal vertex-cover.
Finally, in \Cref{subsec:tardos}, we prove \Cref{thm:tardos}.


\subsection{Strongly maximal matchings}
\label{subsec:matching}

We first construct a hypergraph \(H_1\) with no strongly maximal matching in which all edges have size at most \(3\). We do this by replacing each edge \(e\) of \(H_1^*\) with a gadget \(G_e\) on the vertices of that edge. So
\[H_1=\bigunion \big\{G_e:e\in H_1^*\big\}.\]
Recall that the sets of added vertices \(v_i^\pm(e)\) for each gadget are disjoint: so each new vertex is only used in one edge. This graph will essentially `simulate' the graph \(H_1^*\).

We first make the following observation about the structure of strongly maximal matchings of \(H_1\).
\begin{lemma}\label{lem:gadget-SMmatching}
    Suppose that \(M\) is a strongly maximal matching of \(H_1\). Then for any edge \(e\in H_1^*\) of size \(k=|e|\), the intersection \(M\cap G_e\) either has size \(k-1\), or is exactly the \(k\) outer edges.
\end{lemma}
\begin{proof}
    First note that the outer edges are the unique matching of size at least \(k\). So it suffices to show the intersection has size at least \(k-1\).
    
    Suppose some \(e\) is such that \(|M\cap G_e|<k-1\). Then replacing these edges by the inner edges of \(G_e\) removed fewer than \(k-1\) edges, and adds \(k-1\). Certainly the inner edges of \(G_e\) don't intersect any other edges of \(M\), as no edges outside \(G_e\) meet these vertices. So \(M\) is not a strongly maximal matching, giving a contradiction and thus the desired result.
\end{proof}

We next note that in order to prove \Cref{thm:SMmatching}, it is enough to show that \(H_1\), which is not uniform, has no strongly maximal matching.

\begin{lemma}\label{lem:uniform-SMmatching}
    There is a \(3\)-uniform hypergraph \(H'_1\) and a bijection \(\phi_1:E(H_1)\to E(H'_1)\) such that, for all \(M\sseq E(H_1)\), the image of \(M\) under \(\phi\) is a matching of \(H'_1\) if and only if \(M\) is a matching of \(H_1\).
    
    In particular, \(H_1\) has a strongly maximal matching if and only if \(H'_1\) does.
\end{lemma}
\begin{proof}
    We construct the \(3\)-uniform hypergraph \(H_1'\) from \(H_1\) by adding to each edge \(e\) of size \(2\) a new vertex \(v_e\). For such edges, we set \(\phi_1(e)=e\cup\{v_e\}\), and for edges that already had size \(3\) we take \(\phi_1(e)=e\). 
    But each \(v_e\) is in exactly one edge, so for any two distinct edges \(e,e'\) of \(H_1\), we must have \(\phi_1(e)\cap\phi_1(e')=e\cap e'\). In particular, \(e\) and \(e'\) intersect if and only if \(\phi_1(e)\) and \(\phi_1(e')\) do.
    
    But now \(\phi_1\) and its inverse both preserve matchings. So if a matching \(M\) of \(H_1\) is not strongly maximal, this is witnessed by some other matching \(N\) with \(|N\setminus M|>|M\setminus N|\). But now
    \[\left|\phi_1[N]\setminus\phi_1[M]\right|=|N\setminus M|>|M\setminus N|=\left|\phi_1[M]\setminus\phi_1[N]\right|,\]
    so \(\phi_1[M]\) also cannot be strongly maximal. Hence we have the desired result.
\end{proof}

We now prove \Cref{thm:SMmatching}.

\begin{proof}[Proof of \Cref{thm:SMmatching}]
By \Cref{lem:uniform-SMmatching}, it is enough to prove that \(H_1\) has no strongly maximal matching. Suppose for contradiction that \(M\) is a strongly maximal matching of \(H_1\). Now we define \(M^*\) to be the set of those edges \(e\) of \(H_1^*\) such that \(M\cap G_e\) is exactly the outer edges of \(G_e\). But notice that for each edge \(e\) of \(M^*\), the vertices of \(e\) are all covered in \(M\) by edges of \(G_e\). In particular, no vertex can be in two edges of \(M^*\), so \(M^*\) is a matching. Notice that by \Cref{lem:gadget-SMmatching}, \(M\) contains \(|e|\) edges of \(G_e\) for each \(e\in M^*\), and \(|e|-1\) edges of \(G_e\) otherwise.

Then by \Cref{thm:SMmatching_unbdd}, \(M^*\) cannot be strongly maximal, and so there is a matching \(N^*\) of \(H_1^*\) such that \(|N^*\setminus M^*|>|M^*\setminus N^*|\). We now define a matching \(N\) of \(H_1\) as follows:
\begin{itemize}
    \item For \(e\in M^*\setminus N^*\), we take \(N\cap G_e\) to be exactly the inner edges of \(G_e\). Thus \(N\) uses one fewer edge of \(G_e\) than \(M\).
    \item For any other edge \(e\) of \(H_1^*\) such that \(M\cap G_e\) contains an edge that meets an edge of \(N^*\setminus M^*\), we again take \(N\cap G_e\) to be exactly the inner edges of \(G_e\). Notice that there are at most \(\left|\bigunion (N^*\setminus M^*)\right|\) such edges, and in particular there are finitely many.
    
    Note also that \(e\) is not in \(M^*\), as the only remaining such edges are also in \(N^*\), and so do not meet any other edge of \(N^*\), let alone an edge of \(N^*\setminus M^*\). Thus \(N\) uses the same number of edges of \(G_e\) as \(M\), namely \(|e|-1\).
    \item For \(e\in N^*\setminus M^*\), we take \(N\cap G_e\) to be exactly the \emph{outer} edges of \(G_e\). Now \(N\) uses one more edge of \(G_e\) than \(M\).
    \item For any other edge \(e\) of \(H_1^*\), we take \(N\cap G_e\) to be exactly the same as \(M\cap G_e\).
\end{itemize}
We now have to check that this is still a matching. Every set \(N\cap G_e\) is a matching by definition, so we just need to check that no vertex of \(H_1^*\) is contained in outer edges of \(N\) in two different gadgets \(G_e, G_{e'}\). As \(M\) is a matching, at least one of these outer edges would be in \(N\setminus M\), and so we may assume \(e\in N^*\setminus M^*\).

Then \(e,e'\) must meet, so \(e'\notin N^*\). Hence the only way \(N\) can contain an outer edge of \(G_{e'}\) is if \(N\cap G_{e'}=M\cap G_{e'}\). But now this meets \(e\), an edge of \(N^*\setminus M^*\), so in fact \(N\cap G_{e'}\) is exactly the inner edges of \(G_{e'}\). Thus in fact no vertex can be contained in outer edges of both \(G_e\) and \(G_{e'}\). Hence \(N\) is a matching.

But \(N\) differs from \(M\) on \(G_e\) for finitely many choices of \(e\), so \(N\setminus M\) and \(M\setminus N\) are both finite. And \(N\) contains one more edge of \(G_e\) than \(M\) whenever \(e\in N^*\setminus M^*\); one fewer whenever \(e\in M^*\setminus N^*\); and the same number otherwise. Hence
\[|N\setminus M|-|M\setminus N|=|N^*\setminus M^*|-|M^*\setminus N^*|>0,\]
and so \(M\) is not strongly maximal. Thus \(H_1\) has no strongly maximal matching.
\end{proof}


\subsection{Strongly minimal edge-covers}
\label{subsec:edgecover}

To prove \Cref{thm:SMedgecover}, we use the same method as in \Cref{subsec:matching}, first constructing a graph \(H_2\) in which all edges have size \(2\) or \(3\):
\[H_2=\bigunion \big\{G_e:e\in H_2^*\big\}.\]
Once again, recall that the added vertices \(v_i^\pm(e)\) are each used only in the edge \(e\). We begin with the following observation about strongly minimal edge-covers of \(H_2\).

\begin{lemma}\label{lem:gadget-SMedge}
    Suppose that \(C\) is a strongly minimal edge-cover of \(H_2\). Then for any edge \(e\in H_2^*\) of size \(k=|e|\), the intersection \(C\cap G_e\) either has size \(k\), or is exactly the inner edges.
\end{lemma}
\begin{proof}
There are \(2k-2\) added vertices which are only covered by edges of \(G_e\), each of which covers at most two. So certainly there are least \(k-1\) edges in \(C\cap G_e\). If there are exactly \(k-1\) edges, then each edge must cover exactly two added vertices, and no added vertex can be contained in two such edges. But clearly such a matching must be the \(k-1\) inner edges.

Now suppose that \(C\cap G_e\) contains more than \(k\) edges. Then we can replace these edges by the \(k\) outer edges, as these cover the whole vertex set of \(G_e\). But this would contradict strong minimality, and hence we have the desired result.
\end{proof}

We again note that in order to prove \Cref{thm:SMedgecover}, it is enough to show that \(H_2\), which is not uniform, has no strongly minimal edge-cover.

\begin{lemma}\label{lem:uniform-SMedgecover}
    There is a \(3\)-uniform hypergraph \(H'_2\) and a bijection \(\phi_2:E(H_2)\to E(H'_2)\) such that, for all \(C\sseq E(H_2)\), the image of \(C\) under \(\phi\) is an edge-cover of \(H'_2\) if and only if \(C\) is an edge-cover of \(H_2\).
    
    In particular, \(H_2\) has a strongly minimal edge-cover if and only if \(H'_2\) does.
\end{lemma}
\begin{proof}
    We construct the \(3\)-uniform hypergraph \(H'_2\) from \(H_2\) by adding one new vertex \(v\) to the graph, and adding that vertex to each edge of size \(2\). We take \(\phi_2(e)\) to be \(e\cup \{v\}\) for edges of size \(2\), and take \(\phi_2(e)\) to be \(e\) for edges of size \(3\). Now \(\bigcup\phi_2[C]=(\bigcup C)\cup \{v\}\) if \(C\) contains an edge of size \(2\), and \(\bigcup \phi_2[C]=\bigcup C\) otherwise.
    Hence certainly \(C\) is an edge-cover if \(\phi_2[C]\) is. But now the definition of a gadget implies that some vertex of \(H_2\) is only incident to edges of size \(2\). Hence any edge-cover of \(C\) must contain an edge of size \(2\), and so must have \(\cup\phi_2[C]=V(H_2)\cup \{v\}\). But then \(\phi_2[C]\) is also an edge-cover, giving the desired result.

    But now \(\phi_2\) and its inverse both preserve edge-covers. So just as in the proof of \Cref{lem:uniform-SMmatching}, if an edge-cover \(C\) of \(H_2\) is not strongly minimal, this is witnessed by some other edge-cover \(D\) with \(|D\setminus C|<|C\setminus D|\). But now
    \[\left|\phi_2[D]\setminus \phi_2[C]\right|=|D\setminus C|<|C\setminus D|=\left|\phi_2[C]\setminus \phi_2[D]\right|,\]
    so \(\phi_2[C]\) also cannot be strongly maximal. Hence we have the desired result.
\end{proof}

We now prove \Cref{thm:SMedgecover}, noting that this has many similarities to the proof of \Cref{thm:SMmatching} in the previous section.

\begin{proof}[Proof of \Cref{thm:SMedgecover}]
    By \Cref{lem:uniform-SMedgecover}, it is enough to prove that \(H_2\) has no strongly minimal edge-cover. Suppose for contradiction that \(C\) is a strongly minimal edge-cover of \(H_2\). Now we define \(C^*\) to be the set those edges \(e\) of \(H_2^*\) such that \(C\cap G_e\) contains at least one outer edge of \(G_e\). Now each vertex of \(H_2^*\), viewed as a vertex of \(H_2\), is covered by at least one edge of \(C\), which is necessarily an outer edge of some \(G_e\). But then that vertex is contained in \(e\), an edge of \(C^*\). So \(C^*\) is an edge-cover of \(H_2^*\). Notice that by \Cref{lem:gadget-SMmatching}, \(C\) contains \(|e|\) edges of \(G_e\) for each \(e\in C^*\), and \(|e|-1\) edges of \(G_e\) otherwise.

    Then, by \Cref{thm:SMcover_unbdd}, \(C^*\) cannot be strongly minimal, and so there is a matching \(D^*\) of \(H_2^*\) such that \(|D^*\setminus C^*|<|C^*\setminus D^*|\). We now define an edge-cover \(D\) of \(H_2\) as follows:

    \begin{itemize}
        \item For each \(e\in C^*\setminus D^*\), we take \(D\cap G_e\) to be exactly the inner edges of \(G_e\). Thus \(D\) uses one fewer edge of \(G_e\) than \(C\).
        \item For each \(e\in D^*\setminus C^*\), we take \(D\cap G_e\) to be exactly the outer edges of \(G_e\). Thus \(D\) uses one more edge of \(G_e\) than \(C\).
        \item For each vertex \(v\) contained in an edge of \(C^*\setminus D^*\), note that that vertex is covered by some edge \(e_v\) of \(D^*\) (picking one such edge arbitrarily if there are many options). For each such edge \(e_v\) which was already in \(C^*\), take \(D\cap G_e\) to be exactly the outer edges of \(G_e\). Notice there are finitely many such edges \(e_v\).
        
        Now for \(e=e_v\), the intersection \(D\cap G_e\) certainly covers every vertex that \(C \cap G_e\) did, as well as covering \(v\), and uses the same number of edges of \(G_e\) as \(C\).
        \item For any other edge \(e\) of \(H_2^*\), we take \(D\cap G_e\) to be exactly the same as \(C\cap G_e\).
    \end{itemize}

    We now have to check that this is indeed an edge-cover. By definition, every set \(D\cap G_e\) covers the added vertices of \(G_e\), so we just need to check that the vertices of \(H_2^*\) are covered by \(D\). The only time a vertex \(v\) covered by \(C\cap G_e\) is not also covered by \(D\cap G_e\) is when \(e\in C^*\setminus D^*\). But then either some edge \(e\) of \(D^*\setminus C^*\) contains \(v\), or there is an assigned edge \(e=e_v\). Either way, \(v\) is covered by \(D\cap G_e\), which is exactly the outer edges of that gadget. Thus \(D\) is an edge-cover of \(H_2\).

    But \(D\) differs from \(C\) on \(G_e\) for finitely many choices of \(e\), so \(D\setminus C\) and \(C\setminus D\) are both finite. And \(D\) contains one more edge of \(G_e\) than \(C\) whenever \(e\in D^*\setminus C^*\); one fewer whenever \(e\in C^*\setminus D^*\); and the same number otherwise. Hence
    \[|D\setminus C|-|C\setminus D|=\left|D^*\setminus C^*\right|-\left|C^*\setminus D^*\right|<0,\]
    and so \(S\) is not strongly minimal. Thus \(H_2\) has no strongly minimal edge-cover.
\end{proof}


\subsection{Strongly minimal graph colourings}
\label{subsec:colouring}

We recall that a hypergraph \(H\) is a \emph{flag complex} if every subset of an edge of \(H\) is itself an edge of \(H\), and every set of vertices in which every possible edge of size \(2\) is present in \(H\) is itself an edge of \(H\). 
Recall further that Aharoni and Berger \cite{AB09} conjectured that every flag complex should have a strongly minimal edge-cover. 
However, it is easy to see that the graph \(H_2^+=H_2\cup\{e:e\subset e'\in H_2\}\) is a flag complex, and likewise has no strongly minimal edge-cover: if \(C^+\) were a strongly minimal edge-cover of \(H_2^+\), then we could extend all edges of \(E(H_2^+)\setminus E(H_2)\) in \(C^+\) to be in \(E(H_2)\) by simply adding vertices, producing a strongly minimal edge-cover of \(H_2\).

This will allow us to prove \Cref{thm:SMcolouring}, which we recall states that there is a graph with no strongly minimal colouring.

\begin{proof}[Proof of \Cref{thm:SMcolouring}]
    We define the graph \(G\) on the same vertex set as \(H_2^+\) to contain exactly those edges \(uv\) which are not edges of \(H_2^+\). Now notice that as \(H_2^+\) is a flag complex, the independent sets of \(G\) are exactly the edges of \(H_2^+\).

    Now suppose that \(\chi\from V(G) \to P\) is a strongly minimal colouring of \(G\) for some palette \(P\). Then for each colour \(c \in P\), the set \(\chi^{-1}(c)\) of vertices with that colour is an independent set of \(G\), and thus is an edge \(e_c\) of \(H_2^+\). 
    Note also that \(e_c\) and \(e_{c'}\) are disjoint for any two distinct colours \(c,c'\in P\).

    But now certainly each vertex has some colour, so the edges \(C=\{e_c \st c \in P\}\) form an edge-cover of \(H_2^+\). Then let \(D\) be an edge-cover of \(H_2^+\) that witnesses that \(C\) is not strongly minimal---that is to say, \(C\setminus D=\{e_c:c\in I\}\) is larger than \(D\setminus C=\{e_c:c\in J\}\) for some sets \(I,J\sseq P\). But now we may recolour each vertex \(v\) for which \(\chi(v)\in I\) with some colour \(c\) such that \(v\in e_c\in J\). 
    But then the colouring \(\chi\) could not have been strongly minimal, giving the desired result.
\end{proof}


\subsection{Strongly minimal vertex-covers}
\label{subsec:SMvertexcover}

We now construct a graph with no strongly minimal vertex-cover.

Let $G$ be the graph on vertex set $\NN$ with edge set $E$ defined as follows.
\[E = \set[\big]{\set{u,v}\in \NN^{(2)} \st 2u \leq v}.\]
Assume for contradiction that some $A\sseq \NN$ is a strongly minimal vertex-cover.

Firstly, note that we may assume that $A \neq \NN$, as in this case we can simply remove an element of $A$ to produce a smaller vertex-cover.
Define $x\defined\min(\NN\setminus A)$ to be the minimal element of the complement of $A$. 
If there was a $y\in \NN\setminus A$ with $y \geq 2x$, then $\set{x,y}$ would be an edge of $G$, contradicting the fact that $A$ is a vertex-cover.
Thus we may let $y \defined \max(\NN\setminus A)$ be the largest element not in $A$.

We claim that $A' \defined (A\union\set{x})\setminus\set{y+1,y+2}$ is a vertex-cover of $G$.
Indeed, we know that $y < 2x$, and so $y+2 < 2\min A'$, so the complement of $A'$ induces no edge of $G$, and so $A'$ is a vertex-cover of $G$.
This contradicts the supposed strong minimality of $A$, as required.


\subsection{The problem of Tardos}
\label{subsec:tardos}

We now prove \Cref{thm:tardos} by directly constructing a hypergraph \(H_3\) such that, for every matching \(M\) of \(H_3\), there is a matching \(M'\) of \(H_3\) with \(\card{M\setminus M'} = 1\) and \(\card{M' \setminus M} = 2\).

\begin{proof}
    Let \(H_3\) be the hypergraph with vertex set \(\NN\times\NN\) and edge set \(\set{e_{x,y} \st x,y\geq 1}\), where
    \[e_{x,y} = \set{(x,1),(x,2),\dotsc,(x,y),(x,y),(x+1,y),\dotsc,(2x-1,y),(2x,y)}.\]
    Intuitively, we may think of the \(x\) and \(y\) axes as going to the right and upward respectively, as is standard, and then \(e_{x,y}\) roughly forms the shape of the Greek letter \(\Gamma\) of height \(y\) and width \(x\), where it also meets the horizontal axis at \(x\), as shown in \Cref{fig:tardos}.

    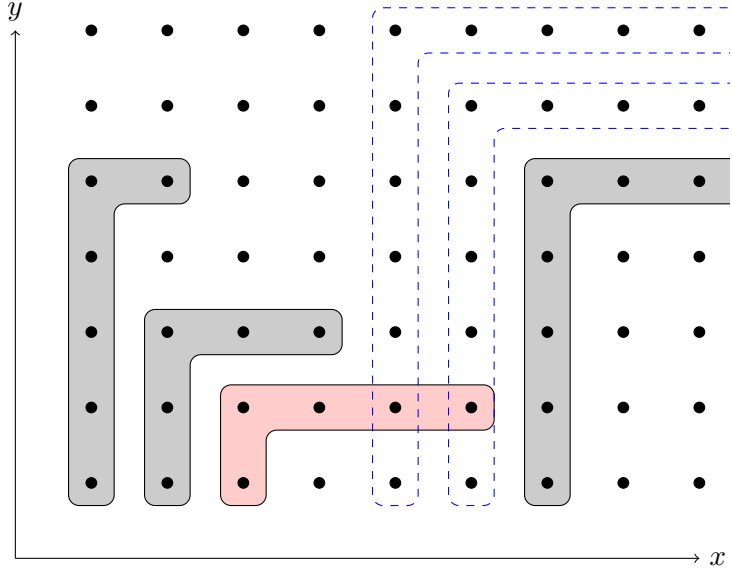
\begin{figure}[h]
\centering
    \begin{tikzpicture}

        \clip (-0.3,-0.3) rectangle + (9.7,8.7);

        \draw[->] (0,0) -- (9,0) node[anchor= west]{\(x\)};
        \draw[->] (0,0) -- (0,7) node[anchor= south]{\(y\)};

        \fill[black!20!white, rounded corners,draw=black] (0.7,0.7) -- (0.7,5.3) -- (2.3,5.3) -- (2.3,4.7) -- (1.3,4.7) -- (1.3,0.7) -- cycle;

        \fill[black!20!white, rounded corners,draw=black] (1.7,0.7) -- (1.7,3.3) -- (4.3,3.3) -- (4.3,2.7) -- (2.3,2.7) -- (2.3,0.7) -- cycle;

        \fill[red!20!white, rounded corners,draw=black] (2.7,0.7) -- (2.7,2.3) -- (6.3,2.3) -- (6.3,1.7) -- (3.3,1.7) -- (3.3,0.7) -- cycle;

        \fill[black!20!white, rounded corners,draw=black] (6.7,0.7) -- (6.7,5.3) -- (14.3,5.3) -- (14.3,4.7) -- (7.3,4.7) -- (7.3,0.7) -- cycle;

        \draw[blue, dashed, rounded corners] (4.7,0.7) -- (4.7,7.3) -- (10.3,7.3) -- (10.3,6.7) -- (5.3,6.7) -- (5.3,0.7) -- cycle;

        \draw[blue, dashed, rounded corners] (5.7,0.7) -- (5.7,6.3) -- (12.3,6.3) -- (12.3,5.7) -- (6.3,5.7) -- (6.3,0.7) -- cycle;

        \foreach \x in {1,...,9}
        {\foreach \y in {1,...,7}
        \filldraw[black] (\x,\y) circle (2pt);
        }

    \end{tikzpicture}
    \caption{A matching in the hypergraph \(H_3\). The edge filled in red can be replaced by the edges drawn in dashed blue.\label{fig:tardos}}
\end{figure}

    Let \(M\sseq E\) be a matching.
    We locate an edge which can be removed, with two added in its place.
    Indeed, assume that \(E = \set{e_{x_1,y_1}, e_{x_2,y_2},\dotsc}\), where \(x_1 < x_2 < \dotsc\), noting that we cannot have two edges with the same value of \(x\), as they would intersect.
    We write \(e_i\) for \(e_{x_i,y_i}\) for short.

    Let \(n\) be minimal such that \(y_{n+1}\geq y_n\), noting that such an \(n\) must exist.
    We will remove edge \(e_n\) from \(M\) and add two edges in its place.
    Define 
    \[A = \set{(x,y) \st x\in \set{2x_n-1,2x_n}, y\in \NN}.\]
    The key claim is that \(e_n\) is the only edge in \(M\) which intersects \(A\).

    Indeed, as \(y_{n+1}\geq y_n\), we must have \(x_{n+1}\geq 2x_n + 1\) for \(e_n\) and \(e_{n+1}\) to be disjoint from each other, and so \(e_{n+1},e_{n+2},\dotsc\) are disjoint from \(A\).
    Moreover, \(x_{n-1} \leq x_n-1\), and so \(e_{n-1}\) does not extend far enough to reach \(A\), and so \(e_{n-1},e_{n-2},\dotsc\) are also all disjoint from \(A\).
    Thus \(e_n\) is the only edge in \(M\) meeting \(A\), as claimed.

    We can now define the edges which will replace \(e_n\) to form a new matching:
    \[M' = (M\setm{e_n})\union\set{e_{2x_n-1,y_{n+1}+2}, e_{2x_n, y_{n+1}+1}}.\]
    Write \(f_1 = e_{2x_n-1,y_{n+1}+2}\) and \(f_2 = e_{2x_n, y_{n+1}+1}\) for convenience.
    It suffices to prove that \(M'\) is a matching, i.e.\ that \(f_1\) and \(f_2\) are disjoint from each other and from all other edges of \(M'\).
    It is immediate that \(f_1\) and \(f_2\) are disjoint from each other, and they are disjoint from \(e_1,e_2,\dotsc,e_{n-1}\) as none of these edges intersect \(A\).
    
    Let \(k\) be minimal such that \(x_{n+k}\geq 4x_n + 1\).
    Then the minimal \(x\)-coordinate of all edges \(e_{n+k},e_{n+k+1},\dotsc\) is at least \(4x_n + 1\), which is larger than the \(x\)-coordinate of any point in \(f_1\) or \(f_2\).
    
    Finally, if \(k\geq 2\), then the minimal \(y\)-coordinate of any edge \(e_{n+1},\dotsc,e_{n+k-1}\) is at most \(y_{n+1}\), as for all \(2\leq j \leq k-1\) we have \((x_{n+j},y_{n+1})\in e_{n+1}\), and so we must have \(y_{n+j} \leq y_{n+1}\) in order that \(e_{n+1}\) and \(e_{n+j}\) be disjoint.
    Thus \(e_{n+j}\) is also disjoint from \(f_1\) and \(f_2\), and so \(M'\) is indeed a matching, as required, and \Cref{thm:tardos} is proved.
\end{proof}


\section{Concluding remarks}

We now discuss some directions of further research, and state several problems which remain open.

Firstly, we may notice that our construction to prove \Cref{thm:tardos} required edges of unbounded size, and so we may refine the question of Tardos to require a hypergraph with bounded edges.
It appears difficult to reduce our counterexample to \Cref{q:tardos} to one with bounded edges by means of a gadget, as we did for \Cref{conj:matching-and-cover}.
Moreover, in the setting of bounded edge-size, the corresponding question for edge-covers also remains open.

\begin{question}
    Does every hypergraph with finitely bounded edges contain a matching $M$ such that no matching $M'$ has $\card{M \setminus M'} = 1$ and $\card{M' \setminus M} = 2$?
\end{question}
\begin{question}
    Does every hypergraph with finitely bounded edges and no isolated vertices contain an edge-cover $C$ such that no edge-cover $C'$ has $\card{C \setminus C'} = 2$ and $\card{C' \setminus C} = 1$?
\end{question}

In a related direction, if \Cref{conj:matching-and-cover} had been true, it would have implied the following refinement of a conjecture of Aharoni and Korman (see \cite[Conjecture 2.2]{Aha22}) by a compactness argument.

\begin{conjecture}
\label{conj:fishbone}
    Let $P$ be a poset of finite width (i.e. there is an integer $w$ such that all antichains of $P$ have size bounded above by $w$).
    Then there exists a chain $C$ in $P$ and a partition of $P$ into antichains such that $C$ meets every antichain in the partition.
\end{conjecture}

This conjecture, also known as the fishbone conjecture, was first stated in 1992 \cite{AK92} with the condition that the antichains were merely of finite size, rather than of size at most $k$.
That version of the conjecture was recently resolved by the first author \cite{Hol24}.
\Cref{conj:fishbone} is known for the case $w=2$ \cite{AK92}, but remains open even for the case of $w=3$.

We also draw attention to the following problem of Pach (see introduction of \cite{Erd94}).

\begin{question}
    Is it true that for any convex body $C$ in \(\RR^d\), there is a packing of congruent copies of $C$ with the property that none of its finite subfamilies can be replaced by a larger system?
\end{question}

Although this problem is much more geometric than those we consider, we nevertheless include it here, as it again concerns a strongly maximal object.

Finally, we conclude with a few words about the case of fractional matchings and covers.
One natural direction to look to weaken the conjectures considered here would be to instead ask whether a hypergraph \(H\) necessarily contains a strongly maximal fractional matching, and similarly for covers.
We recall that a fractional matching is a function \(f\from E(H)\to [0,1]\) such that, for all vertices \(v\), we have \(\sum_{e\ni v} f(e) \leq 1\).
Then \(f\) is strongly maximal if there is no function \(g\from E(H)\to [0,1]\) which differs from \(f\) on only finitely many values and has \(\sum_{e\in E(H)} (g(e) - f(e)) > 0\), i.e.\ has greater total weight on the values on which it differs from \(f\).
Strongly minimal fractional vertex-covers and edge-covers are defined similarly.

One could then ask whether any hypergraph \(H\) with bounded edges necessarily contains a strongly maximal fractional matching, a strongly minimal fractional edge-cover, and a strongly minimal fractional vertex-cover.
However, we remark that it is not too hard to show that \(H_1\), \(H_2\), and \(H_1\) respectively provide examples to show that these objects need not exist.


\bibliographystyle{abbrvnat}  
\renewcommand{\bibname}{Bibliography}
\bibliography{main}


\end{document}